\documentclass[10pt]{amsart}
\usepackage{enumerate,amsmath,amssymb,latexsym,
amsfonts, amsthm, amscd}

\theoremstyle{plain}

\newtheorem{theorem}{Theorem}

\numberwithin{equation}{section}

\newcommand{\ra}{\rightarrow}

\newcommand{\R}{\mathbb{R}}

\addtocounter{section}{-1}

\begin{document}

\title {Cubic Hamiltonians}

\date{}

\author[P.L. Robinson]{P.L. Robinson}

\address{Department of Mathematics \\ University of Florida \\ Gainesville FL 32611  USA }

\email[]{paulr@ufl.edu}

 \keywords{}

\begin{abstract}

We determine a precise necessary and sufficient condition for completeness of the Hamiltonian vector field associated to a homogeneous cubic polynomial on a symplectic plane. 

\end{abstract}

\maketitle

\bigbreak

\section{Introduction}

The flow of the Hamiltonian vector field generated by a smooth function on a symplectic manifold is a familiar object of study. Let the symplectic manifold be simply a symplectic vector space: the Hamiltonian flow generated by a homogeneous linear function is a one-parameter group of translations; the Hamiltonian flow generated by a homogeneous quadratic function is a one-parameter group of linear symplectic transformations. In each of these two cases, the Hamiltonian flow is complete: each maximal integral curve of the Hamiltonian vector field is defined for all time. The case of cubic Hamiltonian functions is different: for some cubics the flow is complete whereas for others it is incomplete. 

\medbreak 

Our primary objective in this paper is to establish a simple necessary and sufficient condition for the cubic $\psi$ on a symplectic plane $(Z, \Omega)$ to generate a complete Hamiltonian flow. In Section 1 we associate with $\psi$ a suitably symmetric linear map from $Z$ to the symplectic Lie algebra ${\rm sp} (Z, \Omega)$; following this map with the determinant yields a quadratic map $\Delta : Z \ra \R$. In Section 2 we analyze an arbitrary integral curve $z : I \ra Z$ of the Hamiltonian vector field $\xi^{\psi}$ defined by $\psi$; we find that the second time-derivative $\stackrel{\circ \circ}{z}$ equals $2 F z$, where the scalar function $F := \Delta \circ z: I \ra \R$ satisfies the equation $\stackrel{\circ \circ}{F} \: =  6 F^2$ familiar from the theory of elliptic functions. In Section 3 we achieve our primary objective, proving that the Hamiltonian vector field $\xi^{\psi}$ is complete if and only if the determinant $\Delta$ is identically zero; beyond this, we comment on the nonconstant integral curves of $\xi^{\psi}$ in the complete case and the incomplete case. Finally, we assemble several remarks on issues arising from the main body of the paper: in particular, we remark that $\Delta$ is identically zero if and only if $\psi$ is a monomial; these remarks we plan to develop more fully in subsequent papers. 

\medbreak 

In a subsequent paper we also plan to present a similar treatment of quartic Hamiltonian functions; for now, we merely note one difference between the cubic case and the quartic case. In the cubic case, the scalar function $F$ satisfies the differential equation $\stackrel{\circ \circ}{F} \: =  6 F^2$ whose elliptic solutions are always Weierstrass Pe functions associated to triangular lattices, with $g_2$ zero; in the quartic case, the corresponding scalar functions include Weierstrass functions associated to rectangular lattices, with $g_2$ nonzero. 

\medbreak

\section{Symplectic Algebra} 

Let $(Z, \Omega)$ be a real symplectic vector space: thus, $Z$ is a vector space and $\Omega : Z \times Z \ra \R$ a nonsingular alternating bilinear form. Though it is not necessary for some of what we shall say, we suppose throughout that $Z$ is two-dimensional, so that $(Z, \Omega)$ is a symplectic {\it plane}. The symplectic algebra ${\rm sp} (Z, \Omega)$ is the (commutator bracket) Lie algebra comprising all linear maps $C: Z \ra Z$ such that for all $x, y \in Z$ 
$$ \Omega (C x, y) + \Omega (x, C y) = 0.$$
As a vector space, ${\rm sp} (Z, \Omega)$ is canonically isomorphic to the space of all symmetric bilinear forms on $Z$: to $C \in {\rm sp} (Z, \Omega)$ there corresponds the symmetric bilinear form 
$$Z \times Z \ra \R : (x, y) \mapsto \Omega (x, C y).$$

\medbreak 

Now, let $\psi : Z \ra \R$ be a homogeneous cubic polynomial. To $\psi$ we associate the (fully) symmetric trilinear function $\Psi : Z \times Z \times Z \ra \R$ with value at $(x, y, z) \in Z \times Z \times Z$ given by 
$$\Psi (x, y, z) = \psi (x + y + z) - \{\psi (y + z) + \psi (z + x) + \psi (x + y)\} + \psi (x) + \psi(y) + \psi(z).$$
When $z\in Z$ is fixed, $\Psi (x, y, z)$ is symmetric bilinear in $(x, y) \in Z \times Z$; it follows that there exists a unique $\Gamma_z \in {\rm sp} (Z, \Omega)$ such that for all $x, y \in Z$ 
$$\Psi (x, y, z) = 2 \Omega (x, \Gamma_z y).$$
Full symmetry of $\Psi$ guarantees that the resulting linear map 
$$\Gamma^{\psi} = \Gamma : Z \ra {\rm sp} (Z, \Omega)$$
is symmetric in the sense that for all $x, y \in Z$ 
$$\Gamma_x y = \Gamma_y x.$$
Note that if $z \in Z$ then 
$$2 \Omega (z, \Gamma_z z) = \Psi (z, z, z) = \{ 27 - (3 \times 8) + 3 \} \psi (z) = 6 \psi (z) $$
or 
$$\psi (z) = \frac{1}{3}  \Omega (z, \Gamma_z z).$$

\medbreak

Differentiation of this formula for $\psi$ yields the result that if $v, z \in Z$ then 
$$\psi_z' (v) =  \frac{1}{3} \{\Omega(v, \Gamma_z z) +  \Omega(z, \Gamma_v z) +\Omega(z, \Gamma_z v)\}$$
whence by symmetry of $\Gamma : Z \ra {\rm sp} (Z, \Omega)$ it follows that 
$$\psi_z' (v) =  \Omega(v, \Gamma_z z).$$
Of course, as $\psi$ is a cubic, the first derivative $\psi_z'$ is quadratic in $z \in Z$. As a bilinear form, the second derivative $\psi_z''$ at $z \in Z$ furnishes another means of introducing $\Psi$ and $\Gamma$: indeed, if also $x, y \in Z$ then 
$$\psi_z'' (y, x) = \Psi (x, y, z) = 2 \Omega (x, \Gamma_z y).$$
This equation represents $\psi_z''$ by $2 \Gamma_z$ relative to the symplectic form $\Omega$; consequently, the classical Hessian of $\psi$ is ${\rm Det}(2 \Gamma_z)$. 

\medbreak 

According to the Cayley-Hamilton theorem, if $z \in Z$ then 
$$\Gamma_z \Gamma_z - ({\rm Tr} \: \Gamma_z) \Gamma_z + ({\rm Det} \: \Gamma_z) I = 0$$
whence the fact that $\Gamma_z \in {\rm sp} (Z, \Omega)$ is traceless implies that
$$\Gamma_z \Gamma_z = - ({\rm Det} \: \Gamma_z) I.$$
We define the scalar function $\Delta^{\psi} = \Delta : Z \ra \R$ by requiring that for each $z \in Z$ 
$$\Delta(z) = - ({\rm Det} \: \Gamma_z) $$
so that 
$$\Gamma_z \Gamma_z = \Delta(z) I.$$

\medbreak 

\begin{theorem} \label{Delta} 
If $z \in Z$ then $\Delta(\Gamma_z z) = \Delta(z)^2.$
\end{theorem} 

\begin{proof} 
If $z = 0$ then both sides of the alleged equation plainly vanish. If $z \neq 0$ then apply the special case $\Gamma_{\Gamma_z z} z = \Gamma_z \Gamma_z z $ of  symmetry repeatedly: a first application gives 
$$\Delta(\Gamma_z z) z = \Gamma_{\Gamma_z z} \Gamma_{\Gamma_z z} z  = \Gamma_{\Gamma_z z} \Gamma_z \Gamma_z z = \Gamma_{\Gamma_z z} \Delta(z) z$$
and a second application gives 
$$\Delta(z) \Gamma_{\Gamma_z z} z = \Delta(z) \Gamma_z \Gamma_z z = \Delta(z) \Delta(z) z = \Delta (z)^2 z$$
whence the alleged equation follows by cancellation. 
\end{proof} 

\medbreak 

\section{Cubic Hamiltonians} 

We shall now view $(Z, \Omega)$ as a symplectic manifold in the natural way. Thus, the vector space $Z$ is naturally a smooth manifold; if $z \in Z$ then there is a natural isomorphism from the vector space $Z$ to the tangent space $T_z Z$ sending $v \in Z$ to the directional derivative operator $v|_z \in T_z Z$ given by the rule that whenever $f: Z \ra \R$ is a smooth map, 
$$v|_z (f) = f_z' (v) = \frac{{\rm d}}{{\rm d} t} f(z + t v)|_{t = 0}.$$
Also, $\Omega$ serves double duty as a nonsingular alternating bilinear form on the vector space $Z$ and as a nonsingular closed two-form on the smooth manifold $Z$; explicitly, if $x, y, z \in Z$ then the value $\Omega_z$ of the two-form at $z$ is given by 
$$\Omega_z ( x|_z, y|_z) = \Omega (x, y). $$

\medbreak

When $f : Z \ra \R$ is a smooth (Hamiltonian) function, the corresponding Hamiltonian vector field $\xi^f \in {\rm Vec} Z$  on $Z$ is defined by the requirement 
$$\xi^f \lrcorner \: \Omega = - {\rm d} f$$
where $\lrcorner$ signifies contraction as usual. An integral curve of the vector field $\xi^f$ is a smooth map $z : I \ra Z$ (on some open interval $I \ni 0$) satisfying the Hamilton equations: for each $t \in I$ the tangent vector to $z$ at $t$ equals the value of $\xi^f$ at $z_t$, thus 
$$\stackrel{\circ}{z}_t \: = \xi^f_{z_t}.$$

\medbreak 

We shall focus on the case of a homogeneous cubic $\psi : Z \ra \R$ as Hamiltonian function. The value of $\xi^{\psi}$ at $z \in Z$ is a vector made tangent at $z$: say 
$$\xi^{\psi}_z = x^{\psi}(z) |_z$$  
with $x^{\psi} : Z \ra Z$ a smooth vector-valued function. Now, let $v, z \in Z$: on the one hand, 
$$(\xi^f \lrcorner \: \Omega)_z (v|_z) = \Omega_z (\xi^{\psi}_z, v|_z) = \Omega_z (x^{\psi}(z) |_z, v|_z) = \Omega (x^{\psi}(z), v) ; $$
on the other hand, 
$$- {\rm d} \psi_z (v|_z) = - \psi'_z (v) = - \Omega (v, \Gamma_z z) = \Omega (\Gamma_z z, v).$$
As the symplectic form $\Omega$ is nonsingular, it follows that 
$$ x^{\psi}(z) = \Gamma_z z.$$ 
Accordingly, the Hamilton equation for $z : I \ra Z$ reads 
$$\stackrel{\circ} {z}  \: = \Gamma_z z.$$ 

\medbreak 

Let $z : I \ra Z$ be a solution of this Hamilton equation. Take a further derivative: as $\Gamma$ is symmetric, 
$$\stackrel{\circ \circ}{z} \: = \Gamma_{\stackrel{\circ}{z}} z + \Gamma_z \stackrel{\circ}{z} \: = 2 \Gamma_z \stackrel{\circ}{z} \: = 2 \Gamma_z \Gamma_z z$$
by a further application of the Hamilton equation. Recall that if $w \in Z$ then $\Gamma_w \Gamma_w = \Delta(w) I$ and write 
$$F : = \Delta \circ z : I \ra \R.$$
It then follows that $z : I \ra Z$ satisfies the second-order equation 
$$\stackrel{\circ \circ}{z} \: = 2 F z.$$ 
Note here that $\Delta$ is defined on the whole space $Z$ while $F$ is defined only along the integral curve $z$. 

\begin{theorem} \label{F}
The scalar function $F$ satisfies the second-order equation 
$$\stackrel{\circ \circ}{F} \: = 6 F^2.$$
\end{theorem} 

\begin{proof} 
From the definition 
$$F I = \Gamma_z \Gamma_z$$
we deduce by repeated differentiation that 
$$\stackrel{\circ}{F} I = \Gamma_{\stackrel{\circ}{z}} \Gamma_z + \Gamma_z  \Gamma_{\stackrel{\circ}{z}}$$
and 
$$\stackrel{\circ \circ}{F} I =  \Gamma_{\stackrel{\circ \circ}{z}} \Gamma_z + 2\Gamma_{\stackrel{\circ}{z}} \Gamma_{\stackrel{\circ}{z}} + \Gamma_z  \Gamma_{\stackrel{\circ \circ}{z}}.$$
Here, the first and last terms on the right both equal $2 F \Gamma_z \Gamma_z = 2 F^2 I$ on account of $\stackrel{\circ \circ}{z} \: = 2 F z$ while $\Gamma_{\stackrel{\circ}{z}} \Gamma_{\stackrel{\circ}{z}}$ equals $F^2 I$ on account of $\stackrel{\circ} {z} \: = \Gamma_z z$ and Theorem \ref{Delta}. 
\end{proof} 

We may at once deduce a first-order integral of this second-order equation: multiply through by $2 \stackrel{\circ}{F}$ to obtain 
$$ 2 \stackrel{\circ}{F} \stackrel{\circ \circ}{F} = 12 F^2 \stackrel{\circ}{F}$$
from which there follows 
$$(\stackrel{\circ}{F} )^2 = 4 F^3 - g_3$$ 
for some real constant $g_3$. This notation is deliberately chosen to accord with the theory of elliptic functions. In fact, the solutions to this first-order differential equation are as follows: \par
$\bullet$ if $g_3$ is nonzero then $F(t) = \wp (t - a)$ for some real $a$ where $\wp$ is the Weierstrass Pe function associated to a triangular lattice (the so-called equianharmonic case); \par 
$\bullet$ if $g_3$ is zero then either $F(t) = (t - a)^{-2}$ for some real $a$ or $F$ is identically zero. 

\medbreak 

Note that when $F$ is a (shifted) Weierstrass Pe function, $\stackrel{\circ \circ}{z} \: = 2 F z$ is a (vectorial) Lam\'e equation and may be solved accordingly; for example, see page 285 of [Forsyth].

\medbreak

\section{Completeness Characterized} 

We continue to let $\Gamma : Z \ra {\rm sp} (Z, \Omega)$ be the symmetric linear map corresponding to the homogeneous cubic $\psi : Z \ra \R$ on the symplectic plane $(Z, \Omega)$; we also continue to let $z : I \ra Z$ be an integral curve of the associated Hamiltonian vector field $\xi^{\psi}$. We shall suppose that the curve $z$ has  initial point $z_0$ and hence initial velocity $\stackrel{\circ}{z}_0 \; = \Gamma_{z_0}  z_0$. Our aim in this section is to decide precisely when such an integral curve may be defined for all time; that is, precisely when the maximal domain of definition $I$ is $\R$ itself. 

\medbreak

The critical case is decided immediately. Let $\xi^{\psi}$ (equivalently, ${\rm d} \psi$) vanish at $z_0$; thus, $z$ has initial velocity $\stackrel{\circ}{z}_0 \; = \Gamma_{z_0}  z_0 = 0$. In this critical case, the solution $z : I \ra Z$ is plainly given by $z_t = z_0$ for all $t \in I$ and the maximal $I$ is indeed $\R$. In this connexion, note further that if an integral curve $z : I \ra Z$ vanishes at any point then so does its velocity vector and hence $z$ itself is identically zero. 

\medbreak 

Now let the integral curve $z : I \ra Z$ be other than critical: thus, $\Gamma_{z_0} z_0 = \; \stackrel{\circ}{z}_0 \; \neq 0$ and of course $z_0 \neq 0$. We distinguish two cases. 

\medbreak 

For the first case, suppose there exists some $s \in I$ such that $0 \neq F(s) = \Delta (z_s)$ and therefore $\stackrel{\circ \circ}{F}(s) = \stackrel{\circ}{F}(s)^2 > 0.$ The comments after Theorem \ref{F} show that $F$ has a double pole at some real $a$; thus $\Gamma_{z_t} \Gamma_{z_t} = F(t)I$ is unbounded as $t \ra a$ and so $z_t$ itself is unbounded as $t \ra a$. In this case, the maximal domain of $z$ omits $a$ and thereby falls short of $\R$. 

\medbreak 

For the second case, suppose that $F(t) = 0$ whenever $t \in I$. Note that the linear map $\Gamma_{z_0}$ kills $\Gamma_{z_0} z_0$ (because $\Gamma_{z_0} \Gamma_{z_0} = F(0) I = 0$) but does not kill $z_0$ (because $\Gamma_{z_0} z_0 = \; \stackrel{\circ}{z}_0 \; \neq 0$); thus $z_0$ and $\stackrel{\circ}{z}_0 $ constitute a basis for the plane $Z$ and so 
$$\{ s(z_0 + t \stackrel{\circ}{z}_0) : s, t \in \R \} = (Z \setminus \R \stackrel{\circ}{z}_0)  \cup \{ 0 \}.$$
The supposition $F \equiv 0$ implies that $\stackrel{\circ \circ}{z} \: = 2 F z \equiv 0$ so that $z_t = z_0 + t \stackrel{\circ}{z}_0$ for all $t \in I$; essentially as in the critical case, the maximal  $I$ is therefore $\R$. Now $\Delta$ vanishes on $z_0 + t \stackrel{\circ}{z}_0$ whenever $t \in \R$ (as $F$ is identically zero) and hence vanishes on $s(z_0 + t \stackrel{\circ}{z}_0)$ whenever $s, t \in \R$ (as $\Delta$ is homogeneous); the continuous function $\Delta$ now vanishes on the dense set $ (Z \setminus \R \stackrel{\circ}{z}_0)  \cup \{ 0 \}$ and therefore vanishes on the whole of $Z$. This proves that if $\Delta$ vanishes on the image of some non-critical integral curve then $\Delta$ vanishes identically. 

\medbreak 

We may now marshal these facts towards the following result. 

\begin{theorem} 
Let $\psi: Z \ra \R$ be a homogeneous cubic and $\Delta^{\psi}$ the associated determinant. \par 
$\bullet$ If $\Delta^{\psi} \equiv 0$ then $\xi^{\psi}$ is complete; each non-constant integral curve is an affine line. \par 
$\bullet$ If $\Delta^{\psi} \not \equiv 0$ then $\xi^{\psi}$ is incomplete; only the constant integral curves are defined for all time. 
\end{theorem} 

\begin{proof} 
If $\Delta \equiv 0$ then each maximal integral curve $z$ has $F \equiv 0$ so that $\stackrel{\circ \circ}{z} \: = 2 F z \equiv 0$ and $z$ on $\R$ is affine, as we have seen. If $\Delta \not \equiv 0$ and the integral curve $z$ is not critical, then $F \not \equiv 0$ so that $z$ experiences finite-time blow-up, as we have seen. 
\end{proof} 

\medbreak 

Looking ahead to the next section, we remark that $\Delta^{\psi}$ is identically zero if and only if $\psi$ is monomial in the sense that there exists $w \in Z$ such that for all $z \in Z$ 
$$\psi(z) = \frac{1}{3} \Omega (w, z)^3.$$

\medbreak 

\section{Remarks} 

In this closing section, we record a number of miscellaneous remarks that stem from the body of this paper.

\medbreak

{\it \footnotesize COORDINATE EXPRESSIONS } 

\medbreak 

Though our whole approach has been intentionally coordinate-free, it is also of interest to see the development in terms of linear symplectic coordinates, not least because this may offer glimpses of a fresh perspective on classical invariant theory. 

\medbreak 

To this end, let $u, v \in Z$ satisfy $\Omega (u, v) = 1$ and so constitute a symplectic basis for $(Z, \Omega)$. Decompose $z \in Z$ as 
$$z = p u + q v$$
with 
$$p = p(z) = \Omega (z, v) , \; \; q = q(z) = \Omega (u, z).$$
Write 
$$a = \Omega (u, \Gamma_u u) , \; \; b = \Omega (u, \Gamma_v u),$$
$$c = \Omega (v, \Gamma_u v) , \; \; d = \Omega (v, \Gamma_v v).$$
With these conventions, the cubic  
$$\psi (z) = \frac{1}{3}  \Omega (z, \Gamma_z z)$$
has coordinate form  
$$\psi(z) = \frac{1}{3} \{ a p^3 + 3 b p^2 q + 3 c p q^2 + d q^3 \}$$
and the (vector) Hamilton equation 
$$\stackrel{\circ} {z}  \: = \Gamma_z z$$
becomes the familiar scalar pair 
$$\stackrel{\circ} {p}  \: = - \frac{\partial \psi}{\partial q}, \; \; \stackrel{\circ} {q}  \: = \frac{\partial \psi}{\partial p}.$$
The associated determinant 
$$\Delta(z) = - ({\rm Det} \; \Gamma_z)$$
assumes the form 
$$\Delta(z) = (b^2 - ac) p^2 + (bc - ad) p q + (c^2 - bd) q^2$$
and is the Hessian of $\psi$ (up to scale). We are not the first to observe that the discriminant 
$$(bc - ad)^2 - 4  (b^2 - ac) (c^2 - bd)$$
of this quadratic is precisely the discriminant 
$$a^2 d^2 - 3 b^2 c^2 + 4 a c^3 + 4 b^3 d - 6 a b c d$$
of the cubic 
$$ a p^3 + 3 b p^2 q + 3 c p q^2 + d q^3;$$
for example, see page 60 of [Salmon]. 

\medbreak 

Of course, a purely coordinate-based approach is possible. Let us indicate partial derivatives more succinctly by means of subscripts. With the cubic 
$$\psi(z) = \frac{1}{3} \{ a p^3 + 3 b p^2 q + 3 c p q^2 + d q^3 \}$$
as above, direct computation reveals that $\psi_{p q} \psi_q - \psi_p \psi_{q q}$ is divisible by $p$ and $\psi_{q p} \psi_p - \psi_q \psi_{p p}$ is divisible by $q$; in each case, the quotient is precisely $2 \{ (b^2 - ac) p^2 + (bc - ad) p q + (c^2 - bd) q^2 \}$ and we recover (twice) the determinant $\Delta$ in coordinate form. In fact, when the Hamilton equations 
$$\stackrel{\circ} {p}  \: = - \psi_q, \; \; \stackrel{\circ} {q}  \: = \psi_q$$ 
are differentiated by time once more, they yield precisely
$$\stackrel{\circ \circ} {p} \; = \psi_{p q} \psi_q - \psi_p \psi_{q q} , \; \; \stackrel{\circ \circ} {q} \; = \psi_{q p} \psi_p - \psi_q \psi_{p p}$$
and we recover the scalar components of $\stackrel{\circ \circ}{z} \: = 2 F z$. 

\medbreak 

{\it \footnotesize CANONICAL FORMS} 

\medbreak 

The simplest type of homogeneous cubic is a monomial: for $w \in Z$ define $\psi^w : Z \ra \R$ by requiring that for all $z \in Z$
$$\psi^w (z) = \frac{1}{3} \Omega (w, z)^3.$$
For this cubic, the corresponding symmetric linear map $\Gamma^w : Z \ra {\rm sp} (Z, \Omega)$ is given by 
$$\Gamma^w_z v = \Omega (z, w) \Omega (w, v) w$$
whenever $z, v \in Z$, and the associated determinant $\Delta^w$ is identically zero. 

\medbreak 

Conversely, let the cubic $\psi$ with corresponding symmetric linear map $\Gamma$ be such that the associated determinant $\Delta$ is identically zero. We claim that $\psi = \psi^w$ for a unique $w \in Z$; to justify this claim, we may of course assume that $\Gamma$ is not itself identically zero. Note that if $z \in Z$ then $\Gamma_z \Gamma_z = 0$ so that ${\rm Ran} \; \Gamma_z \subseteq {\rm Ker} \; \Gamma_z$ with equality precisely when $\Gamma_z \neq 0$. Note also that if $x, y \in Z$ then 
$$\Gamma_x \Gamma_y + \Gamma_y \Gamma_x = \{ \Delta (x + y) - \Delta (x) - \Delta (y) \} I = 0.$$
When $x, y, z \in Z$ let us write 
$$\gamma (x, y, z) = \Gamma_x \Gamma_y z.$$
Observe that this expression is now antisymmetric in its first pair of variables and was already symmetric in its last pair; thus 
$$\gamma (x, y, z) = \gamma (x, z, y) = - \gamma (z, x, y) = - \gamma (z, y, x) = \gamma (y, z, x) = \gamma (y, x, z) = - \gamma (x, y, z)$$
and so $\gamma$ vanishes identically. This proves that if $x, y \in Z$ then 
$${\rm Ran} \; \Gamma_y \subseteq {\rm Ker} \; \Gamma_x$$
and choosing any $z \in Z$ with $\Gamma_z \neq 0$ then gives 
$${\rm Ran} \; \Gamma_z \subseteq \cup_{y \in Z} {\rm Ran} \; \Gamma_y \subseteq \cap_{x \in Z} {\rm Ker} \; \Gamma_x \subseteq {\rm Ker} \; \Gamma_z $$
with equality of the end terms and hence equality throughout, whence 
$$ \cup_{y \in Z} {\rm Ran} \; \Gamma_y = \cap_{x \in Z} {\rm Ker} \; \Gamma_x$$
is a distinguished line in the plane $Z$. Let $w \in Z$ be a basis vector for this line. If $z \in Z$ then $\Gamma_z = \lambda_z ( \cdot ) w$ for some linearly $z$-dependent $\lambda_z$ in the dual $Z^*$: as $\Gamma_z$ kills $w$ so does $\lambda_z$ and therefore $\lambda_z = \mu_z \Omega (w, \cdot )$ for some $\mu_z \in \R$ also linear in $z$; this shows that 
$$\Gamma_z = \mu_z \Omega (w, \cdot ) w$$ 
for some $\mu \in Z^*$. Symmetry of $\Gamma$ forces $\mu$ to kill $w$ so that $\mu = \nu \Omega (\cdot , w)$ for some $\nu \in \R$. In the resulting formula 
$$\Gamma_z = \nu \Omega (z, w) \Omega (w, \cdot) w$$
the cube root of the scalar $\nu$ may be absorbed into $w$; this renders $w$ unique and we conclude that $\Gamma = \Gamma^w$ as claimed. 

\medbreak 

Thus, the assignment $w \mapsto \Gamma^w$ is a (cubic!) bijection from $Z$ to the set of all symmetric linear maps $Z \ra {\rm sp} (Z, \Omega)$ for which the associated determinant $\Delta$ is identically zero. 

\medbreak  

The same conclusion may be reached efficiently (though prosaically) using coordinates. From the identical  vanishing of $\Delta$ in the form 
$$(b^2 - ac) p^2 + (bc - ad) p q + (c^2 - b d) q^2 \equiv 0$$
we deduce (by setting $q = 0, p = 0$, and $p q \neq 0$ in turn) that $b^2 = ac, c^2 = b d$, and $ac = bd$. Let $\lambda$ be the cube root of $a$ and $\mu$ the cube root of $d$: then 
$$(\lambda^2 \mu)^3 = a^2 d = a \cdot ad = a \cdot bc = b \cdot ac = b \cdot b^2 = b^3$$
so that $\lambda^2 \mu = b$ and $\lambda \mu^2 = c$ likewise; it follows that the cubic is a monomial, namely
$$ a p^3 + 3 b p^2 q + 3 c p q^2 + d q^3 = (\lambda p + \mu q)^3.$$
\medbreak 

When the determinant $\Delta$ is not identically zero, there are three possibilities: \par
$\bullet$ $\Delta (z) = 0$ for $z$ on a line-pair through $0$ and $\Delta$ takes values of each sign elsewhere; \par 
$\bullet$ $\Delta (z) = 0$ for $z$ on a line through $0$ and $\Delta$ is positive elsewhere; \par 
$\bullet$ $\Delta(0) = 0$ and $\Delta$ is positive elsewhere;\\ 
and canonical forms may be developed for each of these. In connexion with these possibilities, we remark (from Theorem \ref{Delta}) that if $\Delta$ takes negative values then it also takes positive values. 

\bigbreak 

{\it \footnotesize EVALUATION OF $g_3$ } 

\medbreak 

Let $\psi : Z \ra \R$ be a homogeneous cubic and let the Hamiltonian vector field $\xi^{\psi}$ have $z : I \ra Z$ as an integral curve. As we have seen, $\stackrel{\circ \circ}{z} \: = 2 F z$ where the scalar function $F : I \ra \R$ satisfies $(\stackrel{\circ}{F} )^2 = 4 F^3 - g_3$ for some constant $g_3$ that depends on the integral curve $z$. 

\medbreak 
Let the initial point $z_0$ be such that $\psi(z_0) = 0$; as the Hamiltonian $\psi$ is constant along the integral curve, it follows that $\psi (z_t) = 0$ for all $t \in I$. If $z_0$ itself is zero, then of course $F \equiv 0$ and $g_3 = 0$. Now assume that $z_0$ is nonzero, so that $z_t$ is nonzero for all $t \in I$. For each $t \in I$ we have $0 = 3 \psi (z_t) = \Omega (z_t ,  \stackrel{\circ}{z}_t)$ whence (as $Z$ is a plane) $ \stackrel{\circ}{z}_t$ is parallel to $z_t$; say $ \stackrel{\circ}{z} = \lambda \; z$ for some scalar function $\lambda : I \ra \R$. On the one hand, 
$$F z = \Gamma_z \Gamma_z z = \Gamma_z  \stackrel{\circ}{z} \; = \Gamma_z \lambda z = \lambda \Gamma_z z = \lambda  \stackrel{\circ}{z} \; = \lambda^2 z;$$
on the other hand, 
$$2 F z = \; \stackrel{\circ \circ}{z} \; =  \; \stackrel{\circ}{\lambda} z + \lambda \stackrel{\circ}{z} \; = \; \stackrel{\circ}{\lambda} z + \lambda^2 z = \; \stackrel{\circ}{\lambda} z + F z.$$
Thus 
$$\stackrel{\circ}{\lambda} \; = F = \lambda^2$$
and so 
$$\stackrel{\circ}{F} \ = (\lambda^2)^{\circ} = 2 \lambda \stackrel{\circ}{\lambda} \; = 2 \lambda F = 2 \lambda^3.$$
It follows that in this case, 
$$g_3 = 4 F^3 - (\stackrel{\circ}{F} )^2 = 4 (\lambda^2)^3 - (2 \lambda^3)^2 = 0.$$
In short, an initial point $z_0$ with $\psi( z_0 ) = 0$ spawns an integral curve for which $g_3 = 0$. 

\medbreak

Let us offer some sample computations in coordinates. If $\psi = \frac{1}{3} (p^3 - q^3)$ then $\stackrel{\circ} {p}  \; = - \psi_q = q^2$ and $\stackrel{\circ} {q}  \ = \psi_p = p^2$ so that $ \stackrel{\circ \circ}{p} \; = 2 (pq) p$ and $ \stackrel{\circ \circ}{q} \; = 2 (pq) q$; thus $F = pq$ so $ \stackrel{\circ }{F} \; = F_q  \psi_p  - F_p  \psi_q = p^3 + q^3$ and $g_3 = 4 F^3 -  (\stackrel{\circ}{F} )^2 = 4 p^3 q^3 - (p^3 + q^3)^2 = - (p^3 - q^3)^2$ or $g_3 = - 9 \psi^2 \leqslant 0.$ Similarly, if $\psi = p^2 q + p q^2$ then $F = p^2  + p q + q^2$ and $\stackrel{\circ}{F} \; = (q - p)(2 p + q)(p + 2 q)$; after considerable simplification, $g_3 = 4 F^3 - (\stackrel{\circ}{F} )^2$ yields $g_3 = 27 \psi^2 \geqslant 0.$

\medbreak 

Finally, we remark (without proof - but see page 100 of [Salmon]) that classical invariant theory reappears in general: if 
$$\delta = a^2 d^2 - 3 b^2 c^2 + 4 a c^3 + 4 b^3 d - 6 a b c d$$
denotes the discriminant of the cubic $3 \psi$ then 
$$g_3 = - 9 \delta \: \psi^2$$
so 
$$(\stackrel{\circ}{F} )^2 = 4 F^3 + 9 \delta \: \psi^2.$$

\bigbreak 

\bigbreak

\begin{center} 
{\small R}{\footnotesize EFERENCES}
\end{center} 
\medbreak 

[Forsyth] A.R. Forsyth, {\it Theory of Functions of a Complex Variable}, Cambridge, First Edition (1893). 

\medbreak

[Salmon] G. Salmon, {\it Lessons Introductory to the Modern Higher Algebra}, Dublin, First Edition (1859). 

\medbreak

\end{document}